\documentclass[12pt]{article}
\usepackage{amssymb}
\usepackage{amsmath}
\usepackage[usenames]{color}
\usepackage{mathrsfs}
\usepackage{amsfonts}
\usepackage{amssymb,amsmath}
\usepackage{CJK}
\usepackage{cite}
\usepackage{cases}
\usepackage{amsthm}

\def\cl{\centerline}

\def\al{\alpha}
\def\vs{\vspace*}

\def\Z{\mathbb{Z}}
\def\C{\mathbb{C}}

\def\pa{\partial}
\def\QED{\hfill$\Box$}
\def\pr{\prime}

\def\ni{\noindent}

\numberwithin{equation}{section}
\newtheorem{theo}{Theorem}[section]
\newtheorem{defi}[theo]{Definition}

\newtheorem{lemm}[theo]{Lemma}
\newtheorem{prop}[theo]{Proposition}
\newtheorem{clai}{Claim}
\newtheorem{case}{Case}

\begin{document}
\begin{CJK*}{GBK}{song}

\begin{center}
{\bf\large Loop super-Virasoro Lie conformal superalgebra}
\end{center}

\cl{Xiansheng Dai$^{1,\, 2}$, Jianzhi Han$^{2,\, \dag}$.}\footnote{$^\dag$\ Corresponding author.}

\cl{\small $^{1}$ School of Mathematics Sciences, Guizhou Normal University, Guiyang
550001, China}
\cl{\small  $^{2}$ Shcool of Mathematical Sciences, Tongji University, Shanghai
200092, China}
\cl{\small daisheng158@126.com, jzhan@tongji.edu.cn}
\vs{8pt}

{\small
\parskip .005 truein
\baselineskip 3pt \lineskip 3pt

\noindent{{\bf Abstract:}  The loop super-Virasoro
conformal superalgebra $\mathfrak{cls}$ associated with the loop super-Virasoro algebra is constructed in the present paper. The conformal superderivation algebra of $\mathfrak{cls}$ is  completely determined,  which is   shown to  consist of inner superderivations. And nontrivial free  and free $\Z$-graded  $\mathfrak{cls}$-modules of rank two   are classified. We also give a classification of irreducible free $\mathfrak{cls}$-modules of rank two and all irreducible submodules of each free $\Z$-graded $\mathfrak{cls}$-module of rank two.
\vs{5pt}

\ni{\bf Key words:} loop super-Virasoro algebra, Lie conformal superalgebras,
conformal superderivations, conformal modules.}

\ni{\it Mathematics Subject Classification (2010):} 17B15, 17B40, 17B65, 17B68.}
\parskip .001 truein\baselineskip 6pt \lineskip 6pt
\vspace{5mm}
\section{Introduction}
\setcounter{section}{1}\setcounter{equation}{0}
Lie conformal (super)algebras, originally introduced by Kac in \cite{K1,K2}, encode an axiomatic
description for the singular part of the operator product expansion  of chiral fields in
two-dimensional conformal field theory. They are very closely related to vertex algebras (cf. \cite{B,R}) by the same way as Lie algebras correspond to
their universal enveloping algebras. On the other hand, the theory of Lie conformal (super)algebras give us powerful tools for the
study of infinite-dimensional Lie (super)algebras and associative algebras satisfying the locality property described in \cite{K}.
The conformal (super)algebras have drawn much attention in  branches of   physics and  mathematics since the introduction.
The structure theory, representation theory  and cohomology theory of finite (i.e., finitely generated as $\C[\pa]$-modules)
Lie conformal algebras have been well developed (cf.  \cite{BK, CK,CKW,DK}), and finite simple Lie conformal
superalgebras were  classified  in \cite{FK} and their representation theories were developed in \cite{BKL1,BKL2,KO}.

The object investigated in this paper is a Lie conformal superalgebra   closed related to the loop super-Virasoro algebra $\mathfrak{sl}$ whose structures were studied in \cite{DHS}. It is defined as
a infinite-dimensional Lie superalgebra with a basis $\{L_{\al,i}, G_{\mu,j}\mid \al,i,j\in\Z, \mu\in \frac12+\Z\}$
satisfying the following  commutation relations:
\begin{eqnarray}\label{p2.1}
[L_{\al,i},L_{\beta,j}]=(\al-\beta)L_{\al+\beta,i+j},\
[L_{\al,i},G_{\mu,j}]=(\frac{\al}{2}-\mu)G_{\al+\mu,i+j},\
[G_{\mu,i},G_{\nu,j}]=2L_{\mu+\nu,i+j},
\end{eqnarray} the  even and odd parts of which are $\mathfrak{sl}^{\bar{0}}={\rm span}\{L_{\al,i}\mid\al,i\in\Z\}$ and
 $\mathfrak{sl}^{\bar{1}}={\rm span}\{G_{\mu,j}\mid j\in\Z,\mu\in \frac12+\Z\}$, respectively.
Clearly, $\mathfrak{sl}^{\bar{0}}$ is just the loop Virasoro algebra (cf. \cite{GLZ}) and $ \mathfrak{sl}^{\bar{1}}$ is its module. Hence, the loop super-Virasoro algebra  can be seen as a super extension of the loop Virasoro algebra.

Motivated by the idea from \cite{K2} we associate a Lie conformal superalgebra with the loop super-Virasoro algebra.
It is called the \textit{loop super-Virasoro conformal superalgebra}, denoted by $\mathfrak{cls}$, which is a $\C[\pa]$-module
$\mathfrak{cls}=(\bigoplus_{i\in\Z}\C[\pa]L_i)\oplus(\bigoplus_{i\in\Z}\C[\pa]G_i)$ with a $\C[\partial]$-basis $\{L_i,G_i\mid i\in\Z\}$ satisfying the following  $\lambda$-brackets:
\begin{eqnarray}
&&[L_{i\ \lambda} L_j]=(\partial+2\lambda) L_{i+j},\
[L_{i\ \lambda} G_j]=(\partial+\frac32\lambda) G_{i+j},\label{d1.2}\\ \
&&[G_{i\ \lambda} L_j]=(\frac12\partial+\frac32\lambda) G_{i+j}, \
[G_{i\ \lambda} G_j]=2L_{i+j},  \forall i,j\in\Z. \label{d1.3}
\end{eqnarray} Note that this  is an infinite simple Lie conformal superalgebra,  containing   the loop Virasoro Lie conformal algebra
$\mathfrak{clv}=\bigoplus_{i\in\Z}\C[\pa]L_i$ (see \cite{WCY}) as its subalgebra.
As pointed out previously,  the theory of finite simple Lie conformal (super)algebras were well developed, but so far there  is no systematic theory for  the infinite case. So it is interesting and necessary  to develop the theory for  infinite simple Lie conformal superalgebras. This is one of our motivations for studying the loop super-Virasoro conformal superalgebra. We shall study the superderivation algebra of $\mathfrak{cls}$ and free ($\Z$-graded) $\mathfrak{cls}$-modules of rank $\leq2$. One interesting aspect is that free ($\Z$-graded) $\mathfrak{cls}$-modules of rank 1 are trivial, which is totally different from the loop Virasoro Lie conformal algebra case (all $V_{a,b}$ and $V_{A,b}$ are its nontrivial conformal modules of rank one); and the other interesting aspect is  that the even or odd  part of a $\Z$-grade free $\mathfrak{cls}$-module of rank two has the form $V_{A,b}$ if and only if  the other part is $A_{\frac12,b}$.  We remark that an important strategy frequently used in the present paper is to pass  modules over $\mathfrak{cls}$  to modules over $\mathfrak{clv}$.

This paper is organized as follows. In Sect. 2, we collect some facts and notions related
to Lie conformal superalgebras. In Sect. 3, we determinate conformal superderivations of $\mathfrak{cls}$.
The Sect.  4 is devoted to giving the classification of all nontrivial free $\mathfrak{cls}$-modules of rank less than  two. We also determine the irreducibility of these modules and therefore classify all inequivalent irreducible free  $\mathfrak{cls}$-modules of rank two.
And  free $\Z$-graded $\mathfrak{cls}$-modules of rank less than or equal to two  are  classified in  Sect. 5. Moreover, all irreducible submodules of free $\Z$-graded $\mathfrak{cls}$-modules of rank two are completely determined.

 Throughout this paper, all vector spaces are assumed to be over the complex field $\C$ and all linear maps are $\C$-linear.  The main results of this paper are summarized in Theorems \ref{main3}, \ref{freeofranktwo}, \ref{main4-2} and \ref{main5}.

\section{Preliminaries}
\setcounter{section}{2}\setcounter{equation}{0}
In this section, we recall some notions related to Lie conformal superalgebras and conformal modules from\cite{DK,K1,K2}.

We say that a   vector space $U$ is $\Z_2$-graded if $U=U^{\bar0}\oplus U^{\bar1}$, and  $x\in U^{\bar i}$ is called $\Z_2$-homogenous and write $|x|= i$. For any  two $\Z_2$-graded vector spaces $U$ and  $V$, a linear map $f: U\rightarrow V$ is called homogenous of degree $\bar i\in\Z_2$ if $f(U^{\bar j})\subseteq V^{\overline{i+j}}$ for all $\bar j\in\Z_2$.
\begin{defi}
A Lie conformal superalgebra  is a $\Z_2$-graded $\C[\partial]$-module $\mathcal{A}$ endowed with a linear map
$\mathcal{A}\otimes\mathcal{A}\rightarrow \C[\lambda]\otimes \mathcal{A}, a\otimes b\mapsto [a_\lambda b]$, called the $\lambda$-bracket, and
 satisfying the following axioms $(a, b, c\in \mathcal{A}):$
\begin{eqnarray}
&&[(\partial a)_\lambda b]=-\lambda[a_\lambda b],\quad [a_\lambda b]=-(-1)^{|a||b|}[b_{-\lambda-\partial} a],\label{m2.2}\\
&&[a_\lambda [b_\mu c]]=[[a_\lambda b]_{\lambda+\mu}c]+(-1)^{|a||b|}[b_\mu [a_\lambda c]].
\end{eqnarray}
\end{defi}

It follows from the axioms in (\ref{m2.2}) that
\begin{eqnarray*}
[(f(\partial)a)_\lambda b]=f(-\lambda)[a_\lambda b]\ \mbox{and}\
[a_\lambda (f(\partial)b)]=f(\partial+\lambda)[a_\lambda b],\quad \forall f(\pa)\in\C[\pa].
\end{eqnarray*}

%A \textit{rank} of a Lie conformal superalgebra is its rank as a $\C[\partial]$-module. For example, the super-Virasoro conformal algebra $\mathrm{CSV}$ is a rank two $\C[\partial]$-module generated by two symbols L and G, namely, $\mathrm{CSV}=\C[\partial]L\oplus \C[\partial]G$  and satisfying the following communicate relations:\begin{eqnarray*} [L_\lambda L]=(\partial+2\lambda)L,\ [L_\lambda G]=(\partial+\frac{3}{2}\lambda)G,\ [G_\lambda G]=2L.\end{eqnarray*}
\begin{defi}
A conformal module  over a Lie conformal superalgebra $\mathcal{A}$ or an $\mathcal A$-module  is a $\Z_2$-graded $\C[\partial]$-module  $V$ endowed with a $\lambda$-action
$\mathcal{A}\otimes V\rightarrow \C[\lambda]\otimes V, a\otimes v\mapsto a_\lambda v$, and satisfying the following axioms $(a, b \in \mathcal{A}, v\in V):$
\begin{eqnarray*}&&(\partial a)_\lambda v=-\lambda a_\lambda v,\quad a_\lambda (\partial v)=(\partial +\lambda)a_\lambda v,\\
&&[a_\lambda b]_{\lambda+\mu}v=a_\lambda (b_\mu v)-(-1)^{|a||b|}b_\mu (a_\lambda v).
\end{eqnarray*}
\end{defi}

The rank of an $\mathcal A$-module $V$ is defined to be the rank of $V$ as $\C[\partial]$-module.
\begin{defi}\label{intermod}Let $\mathcal A$ be a Lie conformal superalgebra.
\begin{itemize}\parskip-3pt\item[{\rm (1)}]  $\mathcal{A}$ is called {\it $\Z$-graded} if $\mathcal{A}=\oplus_{i\in \Z}{\mathcal{A}}^i$,  each ${\mathcal{A}}^i$ is a $\C[\partial]$-submodule
 and
$[{\mathcal{A}}^i\,{}_\lambda\, {\mathcal{A}}^j]\subseteq \mathcal{A}^{i+j}[\lambda]$ for any $i,j\in\Z$.

\item[{\rm (2)}] A  conformal module $V$ over $\mathcal A$ is {\it $\Z$-graded} if $V=\oplus_{i\in\Z}V_i$, each $V_i$ is a $\C[\partial]$-submodule and $\mathcal (\mathcal A^i)_\lambda V_j\subseteq V_{i+j}[\lambda]$ for any $i,j\in\Z$.
If each $V_i$ is a free $\C[\partial]$-module of rank $n$, then $V$ is called a free $\Z$-graded $\mathcal A$-module of rank $n$.
\end{itemize}
\end{defi}

Note that the loop super-Virasoro conformal superalgebra $\mathfrak{cls}=(\mathfrak{cls})^{\bar0}\oplus(\mathfrak{cls})^{\bar1}$ is $\Z_2$-graded with $(\mathfrak{cls})^{\bar 0}=\oplus_{i\in\Z}\C[\partial]L_i$ and $(\mathfrak{cls})^{\bar1}=\oplus_{i\in\Z}\C[\partial]G_i$ such that $[{(\mathfrak{cls})^{\alpha}} \ _\lambda (\mathfrak{cls})^{\beta}]\subseteq (\mathfrak{cls})^{\alpha+\beta}[\lambda]$ for any $\alpha,\beta\in\Z_2$, and on the other hand $\mathfrak{cls}=\oplus_{i\in\Z}(\mathfrak{cls})_i$ is also $\Z$-graded with  $(\mathfrak{cls})_i=\C[\partial]L_i\oplus\C[\partial]G_i$ for each $i$ in the sense of Definition \ref{intermod}.

\section{Conformal superderivations }
\setcounter{theo}{0}
 A homogenous linear map $D_\lambda:\mathcal{A}\rightarrow \mathcal{A}[\lambda]$ of degree $\bar i\in\Z_2$
is called a\textit{ homogeneous conformal superderivation of degree} $\bar{i}$ if the following conditions hold:
\begin{eqnarray*}
D_\lambda(\partial a)=(\partial+\lambda)D_\lambda a,\
D_\lambda[a_\mu b]=[[D_\lambda a]_{\lambda+\mu} b]+(-1)^{i|a|}[a_\mu [D_\lambda b]],\quad \forall a,b\in \mathcal A.
\end{eqnarray*} And we write $D$ instead of $D_\lambda$ for simplicity. Denote the set of all conformal superderivations of degree $\alpha\in\Z_2$   by ${\rm CDer}^\alpha(\mathcal{A})$. Then we call ${\rm CDer(\mathcal{A})}={\rm CDer^{\bar0}(\mathcal{A})}\oplus {\rm CDer^{\bar1}(\mathcal{A})}$ the conformal superdirivation algebra of $\mathcal A$ and each element of ${\rm CDer(\mathcal A)}$  a superderivation of $\mathcal A$.
For any $a\in \mathcal{A}$, one can see easily that the linear map $({\rm ad}_a)_\lambda: \mathcal{A}\rightarrow \mathcal{A}[\lambda]$ given by
$({\rm ad}_a)_\lambda b=[a_\lambda b]$ for   $b\in \mathcal{A}$   is  a conformal superderivation, which is called an \textit{inner conformal  superderivation} of $\mathcal A$. Denote the set of all inner conformal superdirivations of $\mathcal A$ by ${\rm CInn(\mathcal A)}$.

Now we are ready to give the main result of this section, which establishes the equality between the two sets ${\rm CDer}(\mathfrak{cls})$ and ${\rm CInn}(\mathfrak{cls})$.

\begin{theo}\label{main3} Every conformal superderivation of $\mathfrak{cls}$ is inner, i.e.,
{\rm CDer}$(\mathfrak{cls})={\rm CInn}(\mathfrak{cls})$.
\end{theo}

\begin{proof}

 Take any $D\in {\rm CDer}(\mathfrak{cls})$.  For this fixed superderivation $D$ and any $c\in\Z$, define $D^c(x_j)=\pi_{c+j}D(x_j)$ for any $j\in\Z$ and $x_j\in\mathcal {(CL)}_j$, where $\pi_c$ is the natural projection from
$\C[\lambda]\otimes \mathfrak{cls}$ to $\C[\lambda]\otimes \mathfrak {(cls)}_c$. Then it is easy to check that $D^c$ is a conformal superderivation of $\mathfrak{cls}$.

We assert that each $D^c$ is inner. For this, we only need to consider the case that $D^c$ is $\Z_2$-homogenous.

\begin{case} $D^c\in {\rm CDer}^{\bar{0}}(\mathfrak{cls})$.
\end{case}
In this case, assume that $$D_\lambda^c(L_i)=f_i(\partial,\lambda)L_{i+c}\quad {\rm and}\quad  D_\lambda^c(G_i)=g_i(\partial,\lambda)G_{i+c}$$for some $f(\partial, \lambda), g_i(\partial, \lambda)\in\C[\partial,\lambda]$.

Applying $D^c_{\lambda}$ to $[L_{0\ \mu} L_i]=(\partial+2\mu) L_{i}$ and comparing the coefficients of $L_{i+c}$ give
\begin{equation*}
(\partial+\lambda+2\mu)f_i(\partial,\lambda)=(\partial +2\lambda+2\mu)f_0(-\lambda-\mu,\lambda)+(\partial+2\mu)f_i(\partial+\mu,\lambda)
\end{equation*}
Setting $\mu=0$ in the formula above, we get
\begin{equation*}
f_i(\partial,i)=(\partial+2\lambda)\frac{f_0(-\lambda,\lambda)}\lambda.
\end{equation*}
Similarly, it follows from $[L_{0\ \mu} G_j]=(\partial+\frac32\mu) G_{j}$ that
\begin{equation}\label{m4.1}
(\partial+\lambda+\frac32\mu)g_{j}(\partial,\lambda)=(\partial+\frac{3(\lambda+\mu)}2)f_0(-\lambda-\mu,\lambda)
+(\partial+\frac{3\mu}2)g_j(\partial+\mu,\lambda).
\end{equation}
Setting $\mu=0$  in (\ref{m4.1}), one has
\begin{equation*}
g_j(\partial,\lambda)=(\partial+\frac{3\lambda}2)\frac{f_0(-\lambda,\lambda)}{\lambda}.
\end{equation*}
Thus, $D^c_\lambda={\rm ad}_{\frac{f_0(\partial,-\partial)}{-\partial}L_c}$.

\begin{case} $D^c\in {\rm CDer}^{\bar{1}}(\mathfrak{cls})$.
\end{case}
Assume that $D^c_\lambda(L_i)=g_i(\partial,\lambda)G_{i+c}$, $D^c_\lambda(G_i)=f_i(\partial,\lambda)L_{i+c}$.
It follows from applying   $D^c_{\lambda}$ to $[L_{0\ \mu} L_i]=(\partial+2\mu) L_{i}$ that
\begin{eqnarray*}
(\partial+\lambda+2\mu)g_i(\partial,\lambda)
=g_0(-\lambda-\mu,\lambda)(\frac12\partial+\frac32(\lambda+\mu))G_{i+c}+g_i(\partial+\mu,\lambda)(\partial+\frac{3\mu}2),
\end{eqnarray*}
from which by setting $\mu=0$ one has
\begin{equation*}
g_i(\partial,\lambda)=(\frac12\partial+\frac32\lambda)\frac{g_0(-\lambda,\lambda)}{\lambda}.
\end{equation*}

Using $[{L_0}\ _\mu G_i]=(\partial+\frac32\mu)G_i$, one has
\begin{eqnarray*}
(\partial+\lambda+\frac{3\mu}2)f_i(\partial,\lambda)
=\big(2g_0(-\lambda-\mu,\lambda)-f_i(\partial+\mu,\lambda)(\partial+2\mu)\big).
\end{eqnarray*}
from which by choosing $\mu=0$ gives \begin{equation*}
f_i(\partial,\lambda)=\frac{2g_0(-\lambda,\lambda)}{\lambda}.
\end{equation*}
Whence one can see that
\begin{equation*}
D^c_\lambda={\rm ad}_{\frac{g_0(\partial,-\partial)}{-\partial}G_c}.
\end{equation*}

So in either case, we see that $D^c={\rm ad}_{x_c}$ for some $x_c\in (\mathfrak{cls})_c$, is inner, completing the assertion. Note from the definition of $D^c$ we see that $D=\sum_{c\in\Z}D^c$. In particular,  $$D(L_0)=\sum_{c\in\Z}{\rm ad}_{x_c}(L_0)=\sum_{c\in\Z, x_c\neq 0}{\rm ad}_{x_c}(L_0)=\sum_{c\in\Z, x_c\neq 0}{\rm ad}_{x_c}(L_0),$$  which must be a finite sum by the fact that  $D$ is a linear map from $\oplus_{i\in\Z}(\mathfrak{cls})_i$ to $\oplus_{i\in\Z}(\mathfrak{cls})_i[\lambda]$. Now it follows from the fact  $0\neq {\rm ad}_{y_c}(L_0)\in (\mathfrak{cls})_c$ for any $0\neq y_c\in (\mathfrak{cls})_c$ that all but finitely many $x_c$ are zero, and therefore $\sum_{c\in\Z}x_c\in\mathfrak{cls}$. This implies $D=\sum_{c\in\Z}{\rm ad}_{x_c}={\rm ad}_{\sum_{c\in\Z}x_c}$ is an inner conformal superderivation.
\end{proof}

\section{Free modules of rank $\leq 2$}
\setcounter{case}{0}
Let $V=\C[\pa]x\oplus \C[\pa]y$ be a free $\C[\pa]$-module of rank two with $V^{\bar{0}}=\C[\pa]x$ and $V^{\bar{1}}=\C[\pa]y$. For any $a,b\in\C$ and $c\in\C^*=\C\setminus\{0\}$, on the one hand,   define the  actions of $L_i$ and  $G_i$  on $V$ as follows:
\begin{eqnarray}\label{m4.9}
&&L_i\ _\lambda x=c^i(\pa+a\lambda+b)x,\ L_i\ _\lambda y=c^i(\pa+(a+\frac12)\lambda+b)y,\nonumber\\
&&G_i\ _\lambda x=c^iy,\ G_i\ _\lambda y=c^i(\pa+2a\lambda+b)x;
\end{eqnarray}
on the other hand,  actions are given by in another way:
\begin{eqnarray}\label{m4.10}
&&L_i\ _\lambda x=c^i(\pa+a\lambda+b)x,\ L_i\ _\lambda y=c^i(\pa+(a-\frac12)\lambda+b)y, \nonumber\\
&&G_i\ _\lambda x=c^i(\pa+(2a-1)\lambda+b)y,\ G_i\ _\lambda y=c^ix.
\end{eqnarray}It is not hard to see that these two different actions can be extended to the whole $\mathfrak{cls}$ such that in both cases $V$ is a $\mathfrak{cls}$-module. Let us denote the former $\mathfrak{cls}$-module by $M_{a,b,c}$ and the latter by $M_{a,b,c}^\prime.$

\begin{prop}\label{prop-5.1}
(i)\ The $\mathfrak{cls}$-module $M_{a,b,c}$ is irreducible if and only if $a\neq 0$, and $M^\prime_{a,b,c}$ is irreducible if and only if $a\neq \frac12$. Moreover, $\C[\partial](\partial+b)x\oplus\C[\partial]y$ and $\C[\partial]x\oplus\C[\partial](\partial+b)y$ are the unique nontrival $\mathfrak{cls}$-submodules of $M_{0,b,c}$ and $M^\prime_{\frac12,b,c}$, respectively.

(ii)\ For any $R,  T\in\{M, M^\prime\}$, then $R_{a,b,c}\cong T_{\alpha,\beta,\gamma}$ if and only if $(a,b,c)=(\alpha,\beta,\gamma)$ and $R=T$.
\end{prop}
\begin{proof}(i)\quad
We only restrict ourself to the irreducibility of $M_{a,b,c}$, the other one can be treated similarly. It is clear that $\C[\partial](\partial+b)x\oplus\C[\partial]y$ is the unique maximal submodule of $M_{0,b,c}$. So the irreducibility of $M_{a,b,c}$ implies that $a\neq 0$.

Conversely, we show that $M_{a,b,c}$ is irreducible if $a\neq 0$. This is equivalent to showing that any submodule $I_{f(\partial), g(\partial)}$ generated by the single nonzero element $f(\partial)x+g(\partial)y$ is the entire $M_{a,b,c}.$ Note by applying the action of $G_i$ on this element if necessary we may assume that $f(\partial)\neq 0$. Due to the irreducibility of $\C[\partial]x$ as an $\mathfrak{clv}$-module (see \cite{DK,WCY}), we can obtain that $$x+h(\partial)y\in I_{f(\partial), g(\partial)}$$ for some $h(\partial)\in\C[\partial]$ by applying the actions of $\mathfrak{clv}$. So $c_i^{-1}G_i\ _\lambda \big(x+h(\partial)y\big)=y+h(\partial+\lambda)(\partial+2a\lambda+b)x\in I_{f(\partial), g(\partial)}[\lambda],$ and then $$y+h(\partial)(\partial+b)x\in I_{f(\partial),g(\partial)}.$$ Now these two elements give $0\neq x-\big({h}(\partial)\big)^2(\partial+b)x\in I_{f(\partial),g(\partial)}.$  It follows from the irreducibility of $\C[\partial]x$ as an $\mathfrak{clv}$-module again that $x\in I_{f(\partial),g(\partial)}$ and therefore $y\in I_{f(\partial),g(\partial)}$ by \eqref{m4.9}. Hence, $I_{f(\partial),g(\partial)}=M_{a,b,c}$, completing the proof.

(ii)\quad Assume that $R_{a,b,c}=\C[\partial]x\oplus \C[\partial]y$ and $T_{\alpha,\beta,\gamma}=\C[\partial]x^\prime\oplus\C[\partial]y^\prime$, and let $$\phi: R_{a,b,c}\rightarrow T_{\alpha,\beta,\gamma}$$ be an isomorphism.  Then there exist some  $f(\partial), g(\partial)\in\C[\partial]$ such that $$\phi(x)=f(\partial)x^\prime\quad {\rm and}\quad  \phi(y)=g(\partial)y^\prime.$$ It follows immediately from $\phi(L_i\ _\lambda x)=L_i\ _\lambda\phi(x)$ that $$c^i(\partial+a\lambda+b)f(\partial)x^\prime=\gamma^i f(\partial+\lambda)(\partial+\alpha\lambda+\beta)x^\prime, $$ from which we can see that $(a,b,c)=(\alpha,\beta,\gamma)$ and that $f(\partial)$ is a constant term, which may be assumed to be $1$ by means of replacing $x^\prime$ by $f(\partial)x^\prime$. It remains to show that $R=T$. Suppose on the contrary that $R\neq T$. Without loss of generality, we assume that $R=M$ and $T=M^\prime$. Then from $\phi(G_i\ _\lambda x)=G_i\ _\lambda\phi(x)$ we obtain that $a=\frac12$.  By (i), $M_{\frac12,b,c}$ is irreducible but $M^\prime_{\frac12,b,c}$ is reducible, so $\phi$ can not be isomorphic, a contradiction.
\end{proof}

The aim of this section is to classify all free $\mathfrak{cls}$-modules of rank two. In fact, the  two classes of $\mathfrak{cls}$-modules constructed as above exhaust  all free $\mathfrak{cls}$-modules of rank two.

\begin{theo}\label{freeofranktwo}
Suppose that $V=V^{\bar{0}}\oplus V^{\bar{1}}$ is a nontrivial free $\mathfrak{cls}$-module of rank two. Then $V$
is either isomorphism to $M_{a,b,c}$ defined by {\rm (\ref{m4.9})} or  $M^{\pr}_{a,b,c}$ defined by {\rm(\ref{m4.10})}.
\end{theo}

Let $V=V^{\bar{0}}\oplus V^{\bar{1}}$ be a nontrivial free $\mathfrak{cls}$-module of rank two. Then both $V^{\bar 0}$ and $V^{\bar 1}$ are nontrivial. To see this, suppose on the contrary that $V^{\bar i}=0$, then $G_j\ _\lambda V^{\overline {i+1}}=0$ and hence $L_j\ _\lambda V^{\overline {i+1}}=0$ for any $j\in\Z$. So in this case $V$ is a trivial $\mathfrak{cls}$-module, a contradiction.
Hence both $V^{\bar{0}}=\C[\partial]x\ {\rm and}\ V^{\bar{1}}=\C[\partial]y$ are $\C[\partial]$ and hence  $\mathfrak{clv}$-modules of rank one.

Observe that $V$ cannot be a trivial $\mathfrak{clv}$-module, since the relations $[L_{i\ \lambda} G_j]=(\partial+\frac32\lambda) G_{i+j}$ for any $i,j\in\Z$ would imply the actions of $G_i$ on $V$ are all trivial.
Now it follows from    \cite[Proposition 4.3]{WCY} that there exist $a, a^\prime, b, b^\prime, c,d\in\C$  such that $$L_i\ _\lambda x=c^i(\pa+a\lambda+b)x,\
L_i \ _\lambda y={d}^i(\pa+a^\prime\lambda+b^\prime)y, \ \forall i\in\Z.$$ Note that at least one of $c$ and $d$ is nonzero, without loss of generality, we assume that $c\neq 0$.

 In order to determine the module structure of $\mathfrak{cls}$ on $V$  we only need to give the explicit  actions of $G_i$ on $x$ and $y$.    Assume \begin{eqnarray*}
 G_i\ _\lambda x=c^ig_i(\partial,\lambda)y\quad {\rm and}
 \quad G_i\ _\lambda y=c^ig^\prime_i(\partial,\lambda)x\quad {\rm for\ some}\ g_i(\partial,\lambda),  g^\prime_i(\partial,\lambda)\in\C[\pa,\lambda].
\end{eqnarray*}

 It follows from $[G_i\ _\lambda G_j]_{\lambda+\mu}x=2L_{i+j}\ _{\lambda+\mu}x$ that
\begin{eqnarray}
&&g_i(\pa+\lambda, \mu)g^{\pr}_j(\pa,\lambda)+g_i(\pa+\mu,\lambda)g^{\pr}_j(\pa,\mu)=\pa+a(\lambda+\mu)+b,\quad \forall i,j\in\Z.\label{d4.5}
\end{eqnarray}
In particular,  setting  $\lambda=\mu$ and  $i=j$ in above formula one has
\begin{equation}\label{m4.7}
g_i(\pa+\lambda,\lambda)g^{\pr}_i(\pa,\lambda)=\pa+2a\lambda+b,
\end{equation}which immediately implies

\begin{equation}
g_i(\pa,\lambda)\neq 0,\ g^{\pr}_i(\pa,\lambda)\neq 0,\ \forall i\in\Z.
\end{equation}
This together with $[G_i\ _\lambda G_j]_{\lambda+\mu}y=2L_{i+j}\ _{\lambda+\mu}y$ forces  $d\neq 0$. So we have arrived at the following lemma.
\begin{lemm}
The polynomials $g_i(\partial,\lambda),  g^\prime_i(\partial,\lambda)$ for all $i\in\Z$ and complex numbers $c, d$ are all nonzero.
\end{lemm}
Using this lemma, we now can give the proof of the theorem above.
\bigskip

\noindent{\em Proof of Theorem \ref{freeofranktwo}}\quad Note that for any $i\in\Z$ by \eqref{m4.7} either ${\rm deg}_{\partial}g_{i}(\partial,\lambda)=1$ and ${\rm deg}_{\partial}g_{i}^\prime(\partial,\lambda)=0$ or ${\rm deg}_{\partial}g_{i}(\partial,\lambda)=0$ and ${\rm deg}_{\partial}g_{i}^\prime(\partial,\lambda)=1$.

\begin{case}
 There exists some $i\in\Z$ such that ${\rm deg}_{\partial}g_{i}(\partial,\lambda)=1$ and ${\rm deg}_{\partial}g_{i}^\prime(\partial,\lambda)=0$. \end{case}

It follows from $[L_j\ _\lambda G_i]_{\lambda+\mu}x=(\pa+\frac32\lambda)G_{j+i}\ _{\lambda+\mu} x$ for any $i,j\in\Z$ that
\begin{eqnarray}
g_i(\pa+\lambda,\mu)(d/c)^j(\pa+a^\prime\lambda+b^\prime)-g_i(\pa,\mu)(\pa+\mu+a\lambda+b)
=(\frac12\lambda-\mu)g_{j+i}(\pa,\lambda+\mu).\label{d4.4}
\end{eqnarray} Now view   terms on both sides of \eqref{d4.4} as polynomials in the variable $\partial$ with   coefficients in $\C[\lambda,\mu]$,  then the coefficient of  $\partial^2$ in the left hand side must be zero, since ${\rm deg}_{\partial} g_{i+j}(\partial, \lambda+\mu)\leq 1$. This forces $d=c$ and hence \eqref{d4.4} turns out to be  \begin{eqnarray}
g_i(\pa+\lambda,\mu)(\pa+a^\prime\lambda+b^\prime)-g_i(\pa,\mu)(\pa+\mu+a\lambda+b)=(\frac12\lambda-\mu)g_{j+i}(\pa,\lambda+\mu).\label{d4.41}\end{eqnarray} Furthermore, by choosing $\lambda=0$ in  \eqref{d4.41} and comparing the coefficients of $\partial\mu$ one can see that ${\rm deg}_{\partial}g_j(\partial,\lambda)=1$ for all $j\in\Z$. This allows us for each $i\in\Z$ to write \begin{eqnarray*}g_i(\partial,\lambda)=\partial s_i(\lambda)+t_i(\lambda)\quad{\rm  for\ some}\  s_i(\lambda),t_i(\lambda)\in\C[\lambda].\end{eqnarray*} Using these expressions, \eqref{d4.41} is equivalent to
\begin{eqnarray}
& s_i(\mu)\big((a^\prime-a)\lambda+b^\prime-b-\mu\big)+\lambda s_i(\mu)=(\frac12\lambda-\mu)s_{i+j}(\lambda+\mu)\label{eq-4.42}\\
&{\rm and} \quad  t_i(\mu)\big((a^\prime-a)\lambda+b^\prime-b-\mu\big)+\lambda(a^\prime\lambda+b^\prime) s_i(\mu)=(\frac12\lambda-\mu)t_{i+j}(\lambda+\mu)\label{eq-4.43}.
\end{eqnarray}
It is not hard to observe from \eqref{eq-4.42} that all  $s_i(\lambda)$ are equal to  a nonzero constant $\gamma$, and also that $a^\prime=a-\frac12$, $b^\prime=b$. Moreover, replacing $y$ by $\gamma y$ we may assume that $\gamma=1$.  At this very moment, \eqref{eq-4.43} can be rewritten as
 \begin{equation}\label{eq-4.43-1}
-(\frac12\lambda+\mu)t_i(\mu)+(a-\frac12)\lambda^2+b\lambda =(\frac12\lambda-\mu)t_{i+j}(\lambda+\mu),\end{equation} from which one can see that  ${\rm deg}t_i(\lambda)\leq 1$ and thus  has the form $t_i(\lambda)=\alpha_i\lambda+\beta_i$ for any $i\in\Z$ and some $\alpha_i,\beta_i\in\C.$  Substituting these explicit expressions into \eqref{eq-4.43-1} and carrying a direct computation  show that $t_i(\lambda)=(2a-1)\lambda+b$ for any $i\in\Z$.

To sum up,  so far under the assumption $$L_i\ _\lambda x=c^i(\pa+a\lambda+b)x$$ we have obtained the following actions\begin{eqnarray*}
 L_i\ _\lambda y=c^i(\pa+(a-\frac12)\lambda+b)y\quad {\rm and}\quad G_i\ _\lambda x=c^i(\pa+(2a-1)\lambda+b)y.
\end{eqnarray*}
Then by the  remark at the very beginning of this proof, $g_i^\prime(\lambda):=g_i^\prime(\partial,\lambda)\in\C[\lambda]$. It follows from $[G_i\ _\lambda G_j] \ _{\lambda+\mu}y=2L_{i+j}\ _{\lambda+\mu}y$  that  $$(\partial+(2a-1)\lambda+b)g_j^\prime(\mu)+(\partial+(2a-1)\mu+b)g_i^\prime(\lambda)=2\big(\partial+(a-\frac12)(\lambda+\mu)+b\big),$$ which gives immediately rise to $g^\prime_i(\lambda)=1$ for all $i\in\Z$. Whence $V\cong M^\prime_{a,b,c}.$

\begin{case}
 There exists some $i\in\Z$ such that ${\rm deg}_{\partial}g_{i}(\partial,\lambda)=0$ and ${\rm deg}_{\partial}g_{i}^\prime(\partial,\lambda)=1$. \end{case}
 For this case,  similar arguments as for the previous one can be adopted to show that $V\cong M_{a,b,c}.$
\QED

\bigskip
We conclude this section with the following classification result, which is a consequence of Proposition \ref{prop-5.1} and Theorem \ref{freeofranktwo}.

\begin{theo}\label{main4-2}(i)
The set $\{M_{a,b,c}, M^\prime_{\alpha,\beta,\gamma}\mid a,b,\alpha,\beta,\in\C\ {\rm and}\  c,\gamma\in\C^* \ {\rm with}\ a\neq 0, \alpha\neq \frac12\}$ is a complete list of inequivalent   irreducible $\mathfrak{cls}$-modules of rank two.

(ii) Any $\mathfrak{cls}$-module $V$ generated by two elements is either trivial or one of the two $\mathfrak{cls}$-modules $M_{a,b,c}$ and $M^\prime_{a,b,c}$ for some $a,b,c\in\C$.
\end{theo}

\begin{proof} (i) is obvious. (ii)  Since any $\C[\partial]$-module generated by two elements is a quotient of some free $\C[\partial]$-module of rank two,   $V$ is a quotient of $M_{a,b,c}$ or $M^\prime_{a,b,c}$ for some  $a,b,c\in\C$ by Theorem \ref{freeofranktwo}. Then (ii) follows from the fact that the nontrivial quotient of $M_{0,b,c}$ or $M_{\frac12,b,c}^\prime$  is a trivial $\mathfrak{cls}$-module by Proposition \ref{prop-5.1} (i).
\end{proof}

\section{Free $\Z$-graded modules of rank $\leq 2$}
\setcounter{section}{5}\setcounter{theo}{0}\setcounter{case}{0}

Let us first collect some results concerning  nontrivial free $\Z$-graded  $\mathfrak{clv}$-modules of rank one. Two classes of such modules were introduced on the free $\C[\partial]$-module $\oplus_{i\in\Z}\C[\partial]v_i$ in \cite{WCY}:
\begin{itemize}\parskip-3pt

 \item[(i)] Denote   by $V_{a,b}$  when $\oplus_{i\in\Z}\C[\partial]v_i$ is equipped with the $\mathfrak{clv}$-module structure given by
 $$L_i\ _\lambda v_j=(\pa+a\lambda+b)v_{i+j},\quad \forall i\in\Z, $$ where $a,b\in\C$;

\item[(ii)] Denote by  $V_{A,b}$  when $\oplus_{i\in\Z}\C[\partial]v_i$ is equipped with the $\mathfrak{clv}$-module structure given by
 \begin{equation*}
L_i\ _\lambda v_j=\begin{cases}(\pa+b)v_{i+j},&\mbox{ if}\ (a_j,a_{i+j})=(0,0),\\
(\pa+\lambda+b)v_{i+j},&\mbox{ if}\ (a_j,a_{i+j})=(1,1),\\
v_{i+j},&\mbox{ if}\ (a_j,a_{i+j})=(0,1),\\
(\pa+b)(\pa+\lambda+b)v_{i+j},&\mbox{ if}\ (a_j,a_{i+j})=(1,0),
\end{cases}\end{equation*}
\end{itemize}where $A=(a_i)_{i\in\Z}$ is an element of the product $\prod_{i\in\Z}\Z_2$ of $\{\Z_2\}_{i\in\Z}$. Note that $V_{\{0\}_{i\in\Z},b}$ coincides with $V_{0,b}$ and $V_{\{1\}_{i\in\Z},b}$ coincides with $V_{1,b}.$

We cite the classification result of nontrivial free $\Z$-graded  $\mathfrak{clv}$-modules of rank one from \cite{WCY}  as a lemma here.

\begin{lemm}\label{lemm-6.1}
Suppose that $V$ is a free $\Z$-graded  $\mathfrak{clv}$-module of rank one. Then $V=V_{a,b}$ or $V=V_{A,b}$ for some $a,b\in\C$ and $A\in\prod_{i\in\Z}\Z_2$.
\end{lemm}

Let $\{x_i\mid i\in\Z\}$ and $\{y_i\mid i\in\Z\}$ be two $\C[\partial]$-linearly  independent sets, and form the $\C[\partial]$-module $V=\bigoplus_{i\in\Z}(\C[\partial]x_i\oplus\C[\partial]y_i)$. For any $a,b\in\C$ and $A\in\prod_{i\in\Z}\Z_2$, we next give four different actions of $\mathfrak{cls}$ on $V$ such that $\oplus_{i\in\Z}\C[\partial]x_i$ and $\oplus_{i\in\Z}\C[\partial]y_i$  have the   form $V_{a,b}$ or $V_{A,b}$ as $\mathfrak{clv}$-modules:

 \begin{eqnarray*}\label{m6.3}
\noindent{\rm (i)}&&\ L_i\ _\lambda x_j=(\pa+a\lambda+b)x_{i+j},\quad\quad\quad\quad L_i\ _\lambda y_j=(\pa+(a-\frac12)\lambda+b)y_{i+j},\nonumber\\
&&\ G_i\ _\lambda x_j=(\pa+(2a-1)\lambda+b)y_{i+j},\quad G_i\ _\lambda y_j=x_{i+j}; \\
{\rm (ii)} &&L_i\ _\lambda x_j=(\pa+a\lambda+b)x_{i+j},\quad\quad \quad \quad ~ L_i\ _\lambda y_j=(\pa+(a+\frac12)\lambda+b)y_{i+j},\nonumber\\
&& G_i\ _\lambda x_j=y_{i+j}, \quad\quad\quad \quad\quad\quad \quad\quad\quad~ ~ G_i\ _\lambda y_j=(\pa+2a\lambda+b)x_{i+j};\\
{\rm (iii)}&&  L_i\ _\lambda x_j=(\pa+\frac12\lambda+b)x_{i+j}\quad({\rm the\ case}\  a=\frac12);\nonumber\\
&&L_i\ _\lambda y_j=\begin{cases}(\pa+b)y_{i+j},&\mbox{ if}\ (a_j,a_{i+j})=(0,0),\\
(\pa+\lambda+b)y_{i+j},&\mbox{ if}\ (a_j,a_{i+j})=(1,1),\\
y_{i+j},&\mbox{ if}\ (a_j,a_{i+j})=(0,1),\\
(\pa+b)(\pa+\lambda+b)y_{i+j},&\mbox{ if}\ (a_j,a_{i+j})=(1,0);\nonumber\\
\end{cases}\\
&& G_i\ _\lambda x_j=\begin{cases}(\pa+b)y_{i+j},&\mbox{ if}\ (a_j,a_{i+j})\in\{(0,0),(1,0)\},\\
y_{i+j},&\mbox{ if}\ (a_j,a_{i+j})\in\{(1,1),(0,1)\};\\
\end{cases}\\
&& G_i\ _\lambda y_j=\begin{cases}(\pa+\lambda+b)x_{i+j},&\mbox{ if}\ (a_j,a_{i+j})\in\{(1,1),(1,0)\},\\
x_{i+j},&\mbox{ if}\ (a_j,a_{i+j})\in\{(0,0),(0,1)\}.\nonumber
\end{cases}
\end{eqnarray*}
\begin{eqnarray*}
{\rm (iv)}
&&L_i\ _\lambda x_j=\begin{cases}(\pa+b)x_{i+j},&\mbox{ if}\ (a_j,a_{i+j})=(0,0),\\
(\pa+\lambda+b)x_{i+j},&\mbox{ if}\ (a_j,a_{i+j})=(1,1),\\
x_{i+j},&\mbox{ if}\ (a_j,a_{i+j})=(0,1),\\
(\pa+b)(\pa+\lambda+b)x_{i+j},&\mbox{ if}\ (a_j,a_{i+j})=(1,0);\nonumber
\end{cases}\\
&&  L_i\ _\lambda y_j=(\pa+\frac12\lambda+b)y_{i+j}\quad({\rm the\ case}\  a=\frac12);\nonumber\\
&& G_i\ _\lambda x_j=\begin{cases}(\pa+\lambda+b)y_{i+j},&\mbox{ if}\ (a_j,a_{i+j})\in\{(1,1),(1,0)\},\\
y_{i+j},&\mbox{ if}\ (a_j,a_{i+j})\in\{(0,0),(0,1)\};\nonumber
\end{cases}\\
&& G_i\ _\lambda y_j=\begin{cases}(\pa+b)x_{i+j},&\mbox{ if}\ (a_j,a_{i+j})\in\{(0,0),(1,0)\},\\
x_{i+j},&\mbox{ if}\ (a_j,a_{i+j})\in\{(1,1),(0,1)\};\\
\end{cases}\\
\end{eqnarray*}

Denote $V$ by $M_{a,b}$,   $M^\prime_{a,b}$,  $M_{A,b}$, $M^\prime_{A,b}$ in the cases (i)-(iv), respectively. One can verify under the given actions,  $M_{a,b}$,   $M^\prime_{a,b}$,  $M_{A,b}$ and $M^\prime_{A,b}$ becomes $\mathfrak{cls}$-modules. Note that they are all $\Z_2$-graded with the even part $\oplus_{i\in\Z}\C[\partial]x_i$ and the odd part $\oplus_{i\in\Z}\C[\partial]y_i$, and also $\Z$-graded with $\C[\partial]x_i\oplus\C[\partial]y_i$ as their $i$-th gradation. So the four classes of modules above are all free $\Z$-graded $\mathfrak{cls}$-modules of rank two.

\begin{theo}\label{main5}
Suppose that $V=V_{\bar{0}}\oplus V_{\bar{1}}$ is a nontrivial  free  $\Z$-graded $\mathfrak{cls}$-module of rank two. Then $V$ is isomorphic to one of the following modules: $$M_{a,b}, M^{\pr}_{a,b}, M_{A,b}, M^\prime_{A,b}, \quad {\rm where}\ a,b\in\C\ {\rm and}\ A\in \prod_{i\in\Z}\Z_2.$$
\end{theo}
\begin{proof}Note that both $V^{\bar0}$ and  $V^{\bar1}$ are nontrivial by the remark after Theorem \ref{freeofranktwo}, hence are nontrivial free $\Z$-graded  $\mathfrak{clv}$-modules.  Let $\{x_i\mid i\in\Z\}$ and  $\{y_i\mid i\in\Z\}$ be $\C[\pa]$-bases of   $V_{\bar{0}}$
and  $V_{\bar{1}}$.  Then by Lemma \ref{lemm-6.1}, $V^{\bar0}=V_{a,b}$ or $V^{\bar0}=V_{A,b}$ for some $a,b\in\C$ and $A\in\prod_{i\in\Z}\Z_2$. Assume that
\begin{eqnarray*}
&&L_i\ _\lambda x_j=f_{i,j}(\partial,\lambda)x_{i+j}, \quad G_i\ _\lambda x_j=g_{i,j}(\partial,\lambda)y_{i+j},\\
&&L_i\ _\lambda y_j=f^\prime_{i,j}(\partial,\lambda)y_{i+j}, \quad G_i\ _\lambda y_j=g^\prime_{i,j}(\partial,\lambda)x_{i+j}
\end{eqnarray*} for any $i,j,k\in\Z$ and some $f(\partial,\lambda), f^\prime(\pa,\lambda), g(\pa,\lambda), g^\prime(\pa,\lambda)\in\C[\pa,\lambda]$. It follows from $[L_i\ _\lambda G_j]_{\lambda+\mu}x_k=(\pa+\frac32\lambda)G_{i+j}\ _{\lambda+\mu}x_k$ that
\begin{eqnarray}
&&g_{j,k}(\pa+\lambda,\mu)f^{\pr}_{i,j+k}(\pa,\lambda)-g_{j,i+k}(\pa,\mu)f_{i,k}(\pa+\mu,\lambda)\nonumber\\
&=&(\frac12\lambda-\mu)g_{i+j,k}(\pa,\lambda+\mu),\forall i,j,k\in\Z.\label{m5.6}
\end{eqnarray}

\begin{clai}  If
$V^{\bar0}=V_{a,b}$  as $\mathfrak{clv}$-modules, then $V^{\bar1}=V_{a-\frac12,b}$, $V_{a+\frac12,b}$ or $V_{A,b}\ (a=\frac12)$. In this  case $V=M_{a,b}$, $M^\prime_{a,b}$ or $M_{A,b}$.
\end{clai}  Assume that $V^{\bar0}=V_{a,b}$.
This entails us to write explicitly the expressions of $f_{i,j}(\partial,\lambda)$ as:  $$f_{i,j}(\partial,\lambda)=\pa+a\lambda+b.$$ It follows from  $[G_i\ _\lambda G_j]_{\lambda+\mu}x_k=2L_{i+j}\ _{\lambda+\mu}x_k$ that
\begin{eqnarray}
g_{j,k}(\pa+\lambda, \mu)g^{\pr}_{i,j+k}(\pa,\lambda)+g_{i,k}(\pa+\mu,\lambda)g^{\pr}_{j,i+k}(\pa,\mu)=2\big(\pa+a(\lambda+\mu)+b\big),\ \forall\ i,j\in\Z.\label{m5.5}
\end{eqnarray}
Setting $j=i$ and $\mu=\lambda$ in  (\ref{m5.5}) gives
\begin{equation}\label{m5.4-1}
g_{i,k}(\pa+\lambda,\lambda)g^{\pr}_{i,i+k}(\pa,\lambda)=\pa+2a\lambda+b,\ \forall i, k\in\Z,
\end{equation}
from which it is not hard to see that for any $i,k\in\Z,$\begin{eqnarray} &{\rm either}\   g_{i,k}(\pa,\lambda):=\mathfrak g_{i,k} \ {\rm is\ a\ constant} \ {\rm and}\  g^{\pr}_{i,i+k}(\pa,\lambda)=\mathfrak g_{i,k}^{-1}(\pa+2a\lambda+b)\nonumber\\
& {\rm or}\ g^\prime_{i,i+k}(\pa,\lambda):= \mathfrak g_{i,i+k}\ {\rm  is\ a\ constant}\  {\rm and}\ g_{i,k}(\pa,\lambda)={\mathfrak g^\prime}^{-1}_{i,i+k}(\pa+(2a-1)\lambda+b).\label{analysis}\end{eqnarray}

\begin{case}\label{c-1}
As $\mathfrak{clv}$-modules, $V^{\bar1}=V_{A,\beta}$  for some $A\in\prod_{i\in\Z}\Z_2\setminus\{\{0\}_{i\in\Z},\{1\}_{i\in\Z}\}$ and $\beta\in\C$.  Then $a=\frac12, \beta=b$ and $V=M_{A,b}$.
\end{case}Note that there are four different expressions for $f^\prime_{i,j}(\pa,\lambda).$  We will check it case by case.

\noindent{(a)}\ There exist some $p,q\in\Z$ such that $f^\prime_{p,q}(\partial,\lambda)=(\partial+\beta)(\partial+\lambda+\beta)$.

Choosing $i=p$,  $k=q-j$  and using $f_{i,k}(\partial,\lambda)=\pa+a\lambda+b$ in \eqref{m5.6} we have
\begin{eqnarray}
&&g_{j,q-j}(\pa+\lambda,\mu)(\partial+\beta)(\partial+\lambda+\beta)-g_{j,p+q-j}(\pa,\mu)(\pa+\mu+a\lambda+b)\nonumber\\
&=&(\frac12\lambda-\mu)g_{p+j,q-j}(\pa,\lambda+\mu), \quad\forall j\in\Z.\label{m5.6-1}
\end{eqnarray}Combining this with \eqref{analysis}, then we must have\begin{eqnarray*}& \beta=b, a=\frac12,  \mathfrak g_{j,q-j}=g_{j, q-j}(\pa+\lambda,\mu)\in\C^*,\\ & g_{j,p+q-j}(\pa,\mu)=\mathfrak g_{j,q-j}(\pa+b) \ {\rm and}\    g_{p+j,q-j}(\pa,\lambda+\mu)=\mathfrak g_{j,q-j}(\partial+b).\end{eqnarray*}
It can be also observed from \eqref{m5.6-1} that all the leading coefficients of  $g_{k,l}(\pa,\lambda)$ are equal (here we regard $g_{k,l}(\pa,\lambda)$ as polynomials in $\partial$ with coefficients in $\C[\lambda]$). By rescaling the basis $\{y_i\mid i\in\Z\}$ we may assume that all the leading coefficients of  $g_{k,l}(\pa,\lambda)$ are equal 1. Hence, \begin{eqnarray*} g_{j, q-j}(\pa,\lambda)=1\quad  {\rm and}\quad g_{j,p+q-j}(\pa,\lambda)=\pa+b,\quad \forall j\in\Z,    \end{eqnarray*} which in turn by using \eqref{analysis} forces $$g_{j, q}^\prime(\pa,\lambda)=\pa+\lambda+b\quad {\rm and}\quad g^\prime_{j, p+q}(\pa,\lambda)=1,\quad\forall j\in\Z.$$ In particular, $$g_{p,q}(\pa,\lambda)=\pa+b\quad{\rm and}\quad g^\prime_{p,q}(\pa,\lambda)=\pa+\lambda+b.$$
Similarly, for any $j\in\Z$ we have:

\noindent{(b)}\ $g_{p,q}(\pa,\lambda)=\pa+b, g^\prime_{p,q}(\pa,\lambda)=1$ if $f^\prime_{p,q}(\partial,\lambda)=\pa+b$;\\
\noindent{(c)}\ $g_{p,q}(\pa,\lambda)=1, g^\prime_{p,q}(\pa,\lambda)=\pa+\lambda+b$ if $f^\prime_{p,q}(\partial,\lambda)=\pa+\lambda+b$;\\
\noindent{(d)}\ $g_{p,q}(\pa,\lambda)= g^\prime_{p,q}(\pa,\lambda)=1$ if $f^\prime_{p,q}(\partial,\lambda)=1$.\\ This is exactly the module $M_{A,b}$, completing Case \ref{c-1}.

 \begin{case}
As $\mathfrak{clv}$-modules $V^{\bar1}=V_{\alpha,\beta}$ for some $\alpha,\beta\in\C$. Then $\alpha=a\pm \frac12$, $\beta=b$ and $V=M_{a,b}$ or $M^\prime_{a,b}$.
 \end{case}
In this case $f_{i,j}^\prime(\pa,\lambda)=\pa+\alpha\lambda+\beta$. So \eqref{m5.6} now becomes \begin{eqnarray*}
&&g_{j,k}(\pa+\lambda,\mu)(\pa+\alpha\lambda+\beta)-g_{j,i+k}(\pa,\mu)(\pa+\mu+a\lambda+b)\nonumber\\
&=&(\frac12\lambda-\mu)g_{i+j,k}(\pa,\lambda+\mu),\forall i,j,k\in\Z.
\end{eqnarray*}Observe from this formula  that the coefficients of $\pa$  in  $g_{j,i+k}(\pa+\lambda,\lambda)$, $g_{j,k}(\pa,\mu)$ and $g_{i+j,k}(\pa,$ $\lambda+\mu)$ are equal, which may be assumed to be 1 for convenience.
Moreover, the degree of $g_{j,k}(\pa,\lambda)$ for any $j,k\in\Z$ is a constant, i.e, deg$_\pa g_{j,k}(\pa,\lambda)=1$ for all $j,k\in\Z$ or   deg$_\pa g_{j,k}(\pa,\lambda)=0$ for all $j,k\in\Z.$ In the former case, $$g_{j,k}(\pa,\lambda)=\pa+(2a-1)\lambda+b, g^\prime_{j,k}(\pa,\lambda)=1, \beta=b\quad {\rm  and}\quad \alpha=a-\frac12,$$ that is, $V=M_{a,b}$; while in the later case, $$g_{j,k}(\pa,\lambda)=1, g^\prime_{j,k}(\pa,\lambda)=\pa+2a\lambda+b, \beta=b\quad{\rm and}\quad \alpha=a+\frac12,$$ that is, $V=M_{a,b}^\prime$.

\begin{clai}\label{clai--2}  The case $V^{\bar{0}}=V_{A,b}$ and $V^{\bar{1}}= V_{B,\beta}$ for some $A,B\in\prod_{i\in\Z}\Z_2\setminus\{\{0\}_{i\in\Z},\{1\}_{i\in\Z}\}$ and $b, \beta\in\C$ can not occur.\end{clai}
Suppose on the contrary that $V^{\bar{0}}= V_{A,b}$ and $V^{\bar{1}}= V_{B,\beta}$. Choosing $i=\mu=0$ in (\ref{m5.6}), one has
\begin{eqnarray*}
g_{j,k}(\pa+\lambda,0)f^{\pr}_{0,k}(\pa,\lambda)-g_{j,k}(\pa,0)f_{0,k}(\pa,\lambda)=\frac12\lambda g_{j,k}(\pa,\lambda).
\end{eqnarray*} While  by definition of the actions of $\mathfrak{clv},$
\begin{eqnarray*}\label{m5.07}f_{0,k}(\pa,\lambda)=\pa+r_k\lambda+b,\ f^{\prime}_{0,k}(\pa,\lambda)=\pa+t_k\lambda+\beta,\quad{\rm where}\ r_k,t_k\in\Z_2.\end{eqnarray*}
Combining these two formulae gives
\begin{eqnarray}\label{m5.8}
g_{j,k}(\pa+\lambda,0)(\pa+t_k\lambda+b^{\prime})-g_{j,k}(\pa,0)(\pa+r_k\lambda+b)=\frac12\lambda g_{j,k}(\pa,\lambda).
\end{eqnarray}
Then $b^{\prime}=b$, by choosing $\lambda=0$. Now we put (\ref{m5.8}) in the following form:
\begin{eqnarray*}
 g_{j,k}(\pa,\lambda)=2(\pa+b)\frac{g_{j,k}(\pa+\lambda,0)-g_{j,k}(\pa,0)}{\lambda}+2rg_{j,k}(\pa+\lambda,0)-2tg_{j,k}(\pa,0).
\end{eqnarray*}Taking the operation $\lim_{\lambda\rightarrow 0}$ on both sides of the above formula we see that $g_{j,k}(\partial,0)$ is  solutions of the ordinary  differential equation
$$2(\pa+b)\frac{d\phi}{d\pa}=(1-2(r_k-t_k))\phi .$$
It is well-known that the general solution of this differential equation is of the form $c(\pa+b)^{\frac{1-2(r_k-t_k)}{2}}$ with $c\in\C$.
 Thus  $g_{j,k}(\pa,0)\in\C[\pa]\cap\C(\pa+b)^{\frac{1-2(r_k-t_k)}{2}}=0$, since $\frac{1-2(r_k-t_k)}{2}$ is  not an integer.
Hence, $g_{j,k}(\pa,0)=0$  and thereby $g_{j,k}(\pa,\lambda)=0$ by (\ref{m5.8}), a contradiction.

\begin{clai}
If $V^{\bar0}=V_{A,b}$ for some $A\in\prod_{i\in\Z}\Z_2\setminus\{\{0\}_{i\in\Z},\{1\}_{i\in\Z}\}$ and $b\in \C$, then $V^{\bar1}=V_{\frac12,b}$ and $V=M^\prime_{A,b}$.
\end{clai}
By Claim \ref{clai--2},  $V^{\bar1}=V_{\alpha,\beta}$ as $\mathfrak{clv}$-modules for some $\alpha,\beta\in\C$. Now it follows from proof of Case \ref{c-1} that $V^{\bar1}=V_{\frac12,b}$ and thereby $V=M^\prime_{A,b}$.
\end{proof}

The irreducible submodules of nontrivial free $\Z$-graded $\mathfrak{cls}$-modules of rank two can be characterized in the proposition.
\begin{prop}We have the following results:
\begin{itemize}
\item[{\rm (1)}]Any   nontrivial  irreducible $\mathfrak{cls}$-submodule of $M_{a,b}$ has the form \begin{eqnarray*}
&\bigoplus_{k\in\Z}\C[\pa](\pa+b)\sum_{i\in I}c_ix_{i+k}\oplus \bigoplus_{k\in\Z}\C[\pa]\sum_{i\in I}c_iy_{i+k}\quad {\rm if}\ a=0,\\
&\bigoplus_{k\in\Z}\C[\pa]\sum_{i\in I}c_ix_{i+k}\oplus \bigoplus_{k\in\Z}\C[\pa](\pa+b)\sum_{i\in I}c_iy_{i+k}\quad {\rm if}\ a=\frac12,\\
&\bigoplus_{k\in\Z}\C[\pa]\sum_{i\in I}c_ix_{i+k}\oplus \bigoplus_{k\in\Z}\C[\pa]\sum_{i\in I}c_iy_{i+k}\quad {\rm if}\ a\neq 0\ {\rm and}\  a\neq \frac12,
\end{eqnarray*}where $I$ is a finite subset of $\Z$ and $(c_i)_{i\in I}$ is a sequence of complex numbers;
\item[{\rm (2)}]Any   nontrivial  irreducible $\mathfrak{cls}$-submodule of $M^\prime_{a,b}$ has the form \begin{eqnarray*}
&\bigoplus_{k\in\Z}\C[\pa](\pa+b)\sum_{i\in I}c_ix_{i+k}\oplus \bigoplus_{k\in\Z}\C[\pa]\sum_{i\in I}c_iy_{i+k}\quad {\rm if}\ a=0,\\
&\bigoplus_{k\in\Z}\C[\pa]\sum_{i\in I}c_ix_{i+k}\oplus \bigoplus_{k\in\Z}\C[\pa](\pa+b)\sum_{i\in I}c_iy_{i+k}\quad {\rm if}\ a=-\frac12,\\
&\bigoplus_{k\in\Z}\C[\pa]\sum_{i\in I}c_ix_{i+k}\oplus \bigoplus_{k\in\Z}\C[\pa]\sum_{i\in I}c_iy_{i+k}\quad {\rm if}\ a\neq 0\ {\rm and}\  a\neq -\frac12,
\end{eqnarray*}where $I$ is a finite subset of $\Z$ and $(c_i)_{i\in I}$ is a sequence of complex numbers;
\item[{\rm (3)}]Any   nontrivial  irreducible $\mathfrak{cls}$-submodule of $M_{A,b}$ has the form
$$\bigoplus_{k\in\Z}\C[\pa]\sum_{i\in I}c_i\delta_{i+k}(\pa)y_{i+k}\oplus\bigoplus_{k\in\Z}\C[\pa]\sum_{i\in I}c_iy_{i+k} ,$$
where $I$ is a finite subset of $\Z$,  $(c_i)_{i\in I}$ is a sequence of complex numbers and $$
\delta_i(\pa)=\begin{cases}(\pa+b),&\mbox{ if}\ a_i=0,\\
1,&\mbox{ if}\ a_i=1;\\
\end{cases}$$

\item[{\rm (4)}]Any   nontrivial  irreducible $\mathfrak{cls}$-submodule of $M^\prime_{A,b}$ has the form
$$\bigoplus_{k\in\Z}\C[\pa]\sum_{i\in I}c_i\delta_{i+k}(\pa)x_{i+k}\oplus \bigoplus_{k\in\Z}\C[\pa]\sum_{i\in I}c_iy_{i+k},$$
where $I$ is a finite subset of $\Z$,  $(c_i)_{i\in I}$ is a sequence of complex numbers and $\delta_i(\pa)$ is defined as in (3).

\end{itemize}
\end{prop}
\begin{proof}
Here we only give the proofs of (1) and (3).  Denote by $X$ and $Y$ the sets $\{x_i\mid i\in\Z\}$ and $\{y_i\mid i\in\Z\},$ respectively. Let $S_{a,b}$ be a nontrivial irreducible $\mathfrak{cls}$-submodule.
For any $0\neq u\in S_{a,b}$,  there exists a finite subset ${\rm Supp}(u)$ of $X\cup Y$ such that $u\in \bigoplus_{z\in {\rm Supp}(u)}\C[\pa]z$.  Take $0\neq u_0\in S_{a,b}$ with $\sharp {\rm Supp}(u_0)$  minimal.  Then we can write  $u_0=\sum_{z\in I_{u_0} }f_z(\pa)z$ for some $f_z(\pa)\in\C[\pa]\setminus\{0\}$.

 (1)\ Consider the case $a\neq0$.  Without loss of generality, we assume that the set ${\rm Supp}(u_0)\cap\{x_i\mid i\in\Z\}$ is nonempty. Fix $z^\prime\in  X\cap {\rm Supp}(u_0).$ By the irreducibility of  $\C[\pa]z^\prime$  as $\C[\pa] L_0$-module \cite{DK},  we may assume that $f_{z^\prime}(\partial)=1.$  Consider the elements \begin{eqnarray*}&&u_0^\prime(c_1,c_2)=L_0\ _\lambda(u_0)|_{\lambda=c_1}-L_0\ _\lambda(u_0)|_{\lambda=c_2}-a(c_1-c_2)u_0\\
 &=&\sum_{z^\prime\neq z\in X\cap {\rm Supp}(u_0)}\Big(f_z(\pa+c_1)(\pa+ac_1+b)-f_z(\pa+c_2)(\pa+ac_2+b)-a(c_1-c_2)f_z(\pa)\Big)z\\
  &&+\sum_{ z\in Y\cap {\rm Supp}(u_0)}\Big(f_z(\pa+c_1)\big(\pa+(a-\frac12)c_1+b\big)-f_z(\pa+c_2)\big(\pa+(a-\frac12)c_2+b\big)\\
  &&~~~~~~~~~~~~~~~~~~~~~~-a(c_1-c_2)f_z(\pa)\Big)z\in S_{a,b}\end{eqnarray*}  where $c_i\in\C.$  Now the minimality of $ {\rm Supp}(u_0)$ implies that $u_0^\prime(c_1,c_2)=0,$ namely, for any $c_i\in\C$ we have
  \begin{eqnarray*}0&=&f_z(\pa+c_1)(\pa+ac_1+b)-f_z(\pa+c_2)(\pa+ac_2+b)\\ &&-a(c_1-c_2)f_z(\pa),\quad \forall z\in X\cap {\rm Supp}(u_0)
 \end{eqnarray*} and \begin{eqnarray*}0&=&f_z(\pa+c_1)(\pa+(a-\frac12)c_1+b)-f_z(\pa+c_2)(\pa+(a-\frac12)c_2+b)\\ &&-a(c_1-c_2)f_z(\pa), \quad\forall z\in Y\cap {\rm Supp}(u_0).\end{eqnarray*} Hence, $0\neq c_z:=f_z(\pa)\in\C\ (c_{z^\prime}=1)$ for each $ z\in X\cap {\rm Supp}(u_0)$ and $f_z(\pa)=0$ for each $z\in Y\cap {\rm Supp}(u_0)$. By the minimality of ${\rm Supp}(u_0)$ again, $Y\cap {\rm Supp}(u_0)$ is empty. That is,  $u_0=\sum_{z\in X\cap {\rm Supp}(u_0)}c_zz\in S_{a,b}.$ Set $I_{u_0}=\{i\mid x_i\in X\cap{\rm Supp}(u_0)\}.$ Then $u_0$ can be written as $u_0=\sum_{i\in I_{u_0}}c_ix_i$ with $c_i=c_{x_i}$. Now one can see that \begin{eqnarray*} S_{a,b}=\bigoplus_{k\in\Z}\C[\pa]\sum_{i\in I_{u_0}}c_ix_{i+k}\oplus \bigoplus_{k\in\Z}\C[\pa](\pa+b)\sum_{i\in I_{u_0}}c_iy_{i+k}\quad &{\rm if}&\ a=\frac12\\
 {\rm and}\ \ \ \ \ S_{a,b}=\bigoplus_{k\in\Z}\C[\pa]\sum_{i\in I_{u_0}}c_ix_{i+k}\oplus \bigoplus_{k\in\Z}\C[\pa]\sum_{i\in I_{u_0}}c_iy_{i+k}\quad &{\rm if}&\ a\neq \frac12.\end{eqnarray*}

Let us assume that $a=0$ and in this case we start with  the nonempty set $Y\cap {\rm Supp}(u_0)$.  Similarly, one can show that $$S_{a,b}=\bigoplus_{k\in\Z}\C[\pa](\pa+b)\sum_{i\in I_{u_0}}c_ix_{i+k}\oplus \bigoplus_{k\in\Z}\C[\pa]\sum_{i\in I_{u_0}}c_iy_{i+k}.$$

(3)\ As for the case $a\neq 0$ in the proof of (1), $u_0=\sum_{i\in I_{u_0}}c_ix_i\in S_{a,b}.$ Then it is not hard to see that $$S_{a,b}=\bigoplus_{k\in\Z}\C[\pa]\sum_{i\in I_{u_0}}c_ix_{i+k}\oplus \bigoplus_{k\in\Z}\C[\pa]\sum_{i\in I_{u_0}}c_i\delta_{i+k}(\pa)y_{i+k}.$$
\end{proof}
\section*{Acknowledgments}
This paper was supported by NSF grants 11371278,~11431010,~11101056,~11501417,~11161010,
the Fundamental Research Funds for the Central Universities of China,
Innovation Program of Shanghai Municipal Education Commission and Program for
Young Excellent Talents in Tongji University.

\end{CJK*}

\begin{thebibliography}{9999}\vskip0pt\small
\parindent=2ex\parskip=-1pt\baselineskip=-1pt
\bibitem{B}R. E. Borcherds, Vertex algebras, Kac-Moody algebras, and the Monster, {\em Proc. Nat. Acad. Sci. U.S.A.} 1986, {\bf83}(10), 3068-3071.
\bibitem{BK} B. Bakalov, V. G. Kac and A. Voronov,  Cohomology of Conformal Algebras, \textit{Commun. Math. Phys.} 1998, \textbf{200}(3): 561--598.
\bibitem{BKL1} C. Boyallian, V. G. Kac, J. I. Liberati and A. Rudakov, Representations of simple finite Lie conformal superalgebras of type $W$ and $S$, \textit{J. Math. Phys.} 2010, \textbf{47}(4): 712--730.
\bibitem{BKL2}  C. Boyallian, V. G. Kac and J. I. Liberati, Classification of Finite Irreducible Modules over the Lie Conformal Superalgebra $CK_6$,
\textit{Commun. Math. Phys.} 2013, \textbf{317}(2): 503--546.

%\bibitem{CHSX}H. Chen, J. Han, Y. Su and Y. Xu, Loop Schr\"{o}dinger-Virasoro Lie conformal algebra,  \textit{Int. J. Math.} 2016, \textbf{27}(6): 1650057.

\bibitem{CK} S.  Cheng and V. G. Kac, Conformal modules, \textit{Asian J. Math.} 1997, \textbf{1}(1): 181--193.

\bibitem{CKW} S.  Cheng, V. G. Kac and M. Wakimoto, \textit{ Extensions of conformal modules, Topological field theory, Primitive forms and related topics}, Birkh$\ddot{a}$user Boston, 1998.
\bibitem{DHS} X. Dai, J. Han and Y. Su, Structures of generalized loop super-Virasoro algebras, \textit{Int. J. Math.} 2015, \textbf{26}(7): 1550041.
\bibitem{DK} A. D'Andrea and V. G. Kac,  Structure theory of finite conformal algebras, \textit{Sel. Math.} 1998, \textbf{4}(3): 377--418.
\bibitem{FK} D. Fattori and V. G. Kac,  Classification of finite simple Lie conformal superalgebras, \textit{J. Algebra} 2002, \textbf{258}(1): 23--59.
\bibitem{GLZ} X. Guo, R. L$\rm\ddot{u}$ and K. Zhao, Simple Harish-Chandra modules, intermediate series modules, and Verma modules
over the loop-Virasoro algebra, {\it Forum Math.}  2011, {\bf 23}: 1029--1052.


\bibitem{K} V. G. Kac, The idea of locality, Physical applications and mathematical aspects of geometry, groups and algebras, Singapore: World Scientific. 1997: 16--32.
\bibitem{K1} V. G. Kac, \textit{Vertex algebras for beginners}, American Mathematical Society Providence Ri, 1997.
\bibitem{K2} V. G. Kac, Formal distribution algebras and conformal algebras, XIIth International Congress in Mathematical Physics, 1997: 80--97.
\bibitem{KO} P. Kolesnikov, Universally defined representations of Lie conformal superalgebras, \textit{J. Symb. Comput.} 2008, \textbf{43}(6): 406--421.
\bibitem{R} M. Roitman,  On Free Conformal and Vertex Algebras, \textit{J. Algebra}, 1998, {\bf 255}(2): 297--323.
\bibitem{WCY} H. Wu, Q. Chen and X. Yue, Loop Virasoro Lie conformal algebra, \textit{J. Math. Phys.} 2014, \textbf{55}(1): 199--219.

\end{thebibliography}
\end{document}